\documentclass[11pt, a4paper, final]{article}

\usepackage[width=14.5cm]{geometry}

\usepackage{amsfonts}
\usepackage{amsmath}
\usepackage{amssymb}
\usepackage{amsthm}
\usepackage{bm}
\usepackage{cite}
\usepackage{showkeys}
\usepackage[english]{babel}
\usepackage{dblaccnt}
\usepackage{accents}
\usepackage{graphicx}
\usepackage{psfrag}
\usepackage{subfig}
\usepackage{amscd}
\usepackage[mathscr]{euscript}
\usepackage{showkeys}

\newtheorem{defi}{Definition}
\newtheorem{lemma}[defi]{Lemma}
\newtheorem{theorem}[defi]{Theorem}
\newtheorem{corollary}[defi]{Corollary}
\newtheorem{proposition}[defi]{Proposition}

\newtheorem{remark}[defi]{Remark}

\renewcommand{\i}{\mathrm{i}}
\renewcommand{\d}[1]{\,\mathrm{d}#1 \,}

\newcommand{\ol}[1]{\overline{#1}}
\newcommand{\K}{\mathcal{K}}
\newcommand{\epsr}{\epsilon_{\mathrm{r}}}
\renewcommand{\epsilon}{\varepsilon}
\newcommand{\loc}{\mathrm{loc}}
\newcommand{\per}{\mathrm{per}}

\newcommand{\RR}{R}

\DeclareMathOperator{\supp}{\mathrm{supp}}
\DeclareMathOperator{\sign}{\mathrm{sign}}

\renewcommand{\Re}{\mathrm{Re}\,}
\renewcommand{\Im}{\mathrm{Im}\,}

\DeclareMathOperator{\curl}{curl}

\renewcommand{\div}{\mathrm{div} \,}

\newcommand*{\N}{\ensuremath{\mathbb{N}}}
\newcommand*{\Z}{\ensuremath{\mathbb{Z}}}
\newcommand*{\R}{\ensuremath{\mathbb{R}}}
\newcommand*{\C}{\ensuremath{\mathbb{C}}}

\begin{document}

\sloppy

\title{Volume Integral Equations for  Scattering from Anisotropic Diffraction Gratings}
\author{Armin Lechleiter\thanks{Center for Industrial Mathematics, University of Bremen, 28359 Bremen, Germany.}
\and Dinh-Liem Nguyen\thanks{DEFI, INRIA Saclay--Ile-de-France and Ecole Polytechnique, 91128 Palaiseau, France.}} 

\maketitle

\begin{abstract}
  We analyze electromagnetic scattering of TM polarized waves from a diffraction grating 
  consisting of a periodic, anisotropic, and possibly negative-index dielectric material.
  Such scattering problems are important for the modelization of, e.g., light propagation 
  in nano-optical components and metamaterials.  
  The periodic scattering problem can be reformulated as a strongly singular volume integral 
  equation, a technique that attracts continuous interest in the engineering community, but 
  rarely received rigorous theoretic treatment. 
  In this paper we prove new (generalized) G\r{a}rding inequalities in weighted and 
  unweighted Sobolev spaces for the strongly singular integral equation.
  These inequalities also hold for materials for which the real part
  takes negative values inside the diffraction grating, independently of the 
  value of the imaginary part. 
\end{abstract}

\section{Introduction}

We consider scattering of time-harmonic electromagnetic 
waves from diffraction gratings. These three-dimensional dielectric
structures are periodic in one spatial direction and invariant in a 
second, orthogonal, direction (compare Figure~\ref{fig:0}). They 
are used as optical components, e.g., to split up light into beams 
with different directions, and they serve in optical devices as, e.g., 
monochromators or as optical spectrometers.

\begin{figure}[h!!!!tb]
  \begin{center}
    \psfrag{x1}{$x_1$}
    \psfrag{x2}{$x_2$}
    \psfrag{x3}{$x_3$}
    \includegraphics[width=7cm]{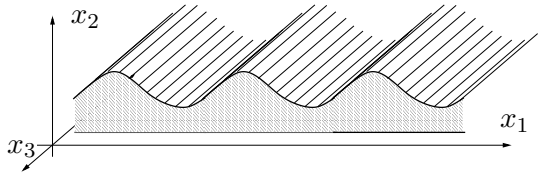}
    \caption{The diffraction 
    grating is periodic in $x_1$, invariant in $x_3$ and bounded in $x_2$.}
  \label{fig:0}
  \end{center}
\end{figure}

If the wave vector of an incident electromagnetic plane wave is chosen 
perpendicular to the invariance direction of the grating, Maxwell's equations 
decouple into scalar Helmholtz equations, known as transverse magnetic 
(TM) and transverse electric (TE) modes 
(these terms are not consistently used in the literature). In this 
paper, we consider the equation of the TM mode for a non-magnetic grating, 
\[
  \div ( A \nabla u) + k^2  u = 0, \qquad k>0,
\]
for an $\alpha$-quasi-periodic function $u$
(that is, $u(x_1 + 2\pi,x_2) = \exp(2\pi \i \alpha) u(x_1,x_2))$). 
In particular, we allow 
the real part of the discontinuous and matrix-valued material parameter $A$ 
to be negative-definite inside the grating structure, independently of the values 
of the imaginary part. (See below for the definition of the real and imaginary 
-- or: self- and non-selfadjoint -- parts of a matrix.) Negative definite material 
parameters are a feature that arises in the modelization of, e.g., optical metamaterials, 
but also for metals at certain frequencies, see, e.g.,~\cite{Shala2007}. 

We reformulate the scattering problem using ($\alpha$-quasi-periodic) volume 
integral equations. Those turn out to be strongly singular and do not 
fit into the standard Riesz theory, since the integral operators are not compact. 
Nevertheless, we prove G\r{a}rding inequalities for the integral equations in 
weighted $\alpha$-quasi-periodic Sobolev spaces, 
which yields a Fredholm framework for the scattering problem. 
This result even holds if the real part $\Re (A)$ of the material parameter 
is negative definite inside the grating, independently of the imaginary part $\Im(A)$.
Our approach extends a technique from~\cite{Kirsc2009}, where similar 
volume integral equations have been analyzed for free space scattering 
problems in case that the scalar real-valued contrast is strictly positive. 
Moreover, we also prove that the G\r{a}rding inequalities in 
weighted Sobolev spaces can be transformed to inequalities in standard 
$\alpha$-quasi-periodic Sobolev spaces, if the grating consists of isotropic material. 
Such standard G\r{a}rding inequalities are particularly useful for 
numerical approximation since the corresponding Galerkin methods are 
easier to implement in standard Sobolev spaces than in weighted spaces. 

In the engineering community, volume integral equations are a popular 
tool to numerically solve scattering problems, see, 
e.g.,~\cite{Richm1965, Richm1966}, since they allow to solve problems 
with complicated material parameters via one single integral equation. 
However, an analysis of this technique in suitable function spaces is 
usually missing, in particular when the material parameters are 
not globally smooth, and when the arising integral operators are not compact.

Recently, volume integral equations started to attract considerable interest 
in applied mathematics~\cite{Vaini2000, Hohag2001, Potth1999, Kirsc2009, Costa2010, Costa2011}.
The papers~\cite{Potth1999, Kirsc2009, Costa2010} analyze strongly 
singular integral equations for scattering in free space.
However,~\cite{Potth1999} considers media with globally 
continuous material properties, and the $L^2$-theory in~\cite{Kirsc2009}
does not yield physical solutions if the material parameter appearing
in the highest-order coefficients are not smooth. 
The paper~\cite{Costa2010} proves a G\r{a}rding inequality for a 
strongly singular volume integral equation arising from electromagnetic 
scattering from a (discontinuous) dielectric. Finally,~\cite{Costa2011} 
determines the essential 
spectrum of strongly singular volume integral operators arising in 
electromagnetic scattering for constant contrast.

The analysis of the integral equation for material parameters 
with negative real part is, to the best of our knowledge, the first 
application of $T$-coercivity (a well-known framework for variational 
formulations of elliptic partial differential equations with sign-changing coefficients, 
see~\cite{Bonne2007, Bonne2010, Bonne2011}) 
to volume integral equations. As usual, the material parameter is, 
however, not allowed to take arbitrary negative values; the solvability 
condition for instance excludes that the relative material parameter 
takes the value $-1$ inside the grating.

The paper is organized as follows: In Section~\ref{se:direct} we briefly 
recall variational theory for the direct scattering problem. In 
Section~\ref{se:integral} we derive the $\alpha$-quasi-periodic volume integral equation.
In Sections~\ref{se:isotropic} and~\ref{se:anisotropic} we prove G\r{a}rding 
inequalities in Sobolev spaces for this equation. The two appendices contain 
two well-known results that do not fit comfortably into the main body of the text.

\emph{Notation:} The usual $L^2$-based Sobolev and Lipschitz spaces on a domain 
$\Omega$ are denoted as $H^s(\Omega)$ and $C^{n,1}(\ol{\Omega})$, respectively. 
Further, $H^s_{\loc}(\Omega) = \{ v\in H^s(B) \text{ for all open balls } B \subset \Omega \}$.
As usual, real and imaginary parts of a 
square matrix $A$ are defined by $\Re A = (A+A^\ast)/2$ and 
$\Im A = (A - A^\ast) / (2\i)$, where $A^\ast$ denotes the
transpose conjugate matrix. Both $\Re A$ and $\Im A$ are self-adjoint 
and  $A = \Re(A) + \i \Im (A)$. We denote the absolute value and the 
Euclidean norm by $|\cdot|$, and the spectral matrix norm by $|\cdot |_2$.

\section{Problem Setting}
\label{se:direct}

Propagation of time-harmonic electromagnetic waves in an inhomogeneous
and isotropic medium without free currents is described by the time-harmonic 
Maxwell's equations for the electric and magnetic fields $E$ and $H$, respectively, 
\begin{equation}
\label{eq:maxwellEquation}
 \curl H + \i\omega\varepsilon E = \sigma E, \qquad \curl E - \i\omega\mu_0 H = 0,
\end{equation}
where $\omega>0$ denotes the angular frequency, $\varepsilon$  is the positive electric 
permittivity, $\mu_0$ is the (scalar, constant and positive) magnetic permeability, 
and $\sigma$ is the conductivity. The permittivity and conductivity are allowed to be
anisotropic, but required to be of the special form 
\[
   \epsilon = \left( \begin{matrix} \epsilon_T & 0 \\ 0 & \epsilon_{33} \end{matrix} \right),
   \qquad
   \sigma = \left( \begin{matrix} \sigma_T & 0 \\ 0 & \sigma_{33} \end{matrix} \right),
 \]
with real and symmetric $2\times2$ matrices $\epsilon_T = (\epsilon_{ij})_{i,j=1,2}$ 
and $\sigma_T = (\sigma_{ij})_{i,j=1,2}$, and real functions $\epsilon_{33}$ and 
$\sigma_{33}$. Furthermore, we assume in this paper that all three 
material parameters are independent of the third variable $x_3$ and $2\pi$-periodic 
in the first variable $x_1$. Moreover, $\epsilon$ equals $\epsilon_0 I_3 >0$ 
(where $I_n$ is the $n \times n$ unit matrix) and $\sigma$
equals zero outside the grating.

If an incident electromagnetic plane wave
independent of the third variable $x_3$ 
illuminates the grating, then 
Maxwell's equations~\eqref{eq:maxwellEquation} for 
the total wave field decouple into two scalar 
partial differential equations (see, e.g.,~\cite{Elsch1998}). 
Indeed, since both, $E$ and $H$ do not depend on $x_3$
it holds that 
$\curl E=(\partial_2 E_3,-\partial_1 E_3, \partial_1 E_2 -\partial_2 E_1)^\top$
and $\curl H=(\partial_2 H_3,-\partial_1 H_3, \partial_1 H_2 -\partial_2 H_1)^\top$.
If we plug these two relations into  the Maxwell's 
equations~\eqref{eq:maxwellEquation} we find that $H_3$ 
satisfies the two-dimensional scalar and anisotropic equation
\begin{equation}
 \label{eq:HMode}
 \div \left( \epsr^{-1} \nabla u\right) + k^2 u = 0 \qquad \text{in } \R^2
\end{equation}
with wave number $k := \omega\sqrt{\epsilon_0 \mu_0}$
and material parameter 
\[
 \epsr := \epsilon_0^{-1} \left[ \left( \begin{smallmatrix} \epsilon_{22}  & - \epsilon_{21} \\ - \epsilon_{12} & \epsilon_{11} \end{smallmatrix} \right) 
 + \i \left( \begin{smallmatrix} \sigma_{22} & - \sigma_{21} \\ - \sigma_{12} & \sigma_{11} \end{smallmatrix} \right) / \omega \right].
\]
The usual jump conditions for the Maxwell's equations imply that the field 
$u$ and the co-normal derivative $\nu \cdot \epsr^{-1} \nabla u$ are continuous across 
interfaces with normal vector $\nu$ where $\epsr$ jumps.
Note that $\epsilon_r$ is $2\pi$-periodic in 
$x_1$ and equals $I_2$  outside the grating. 

We seek for weak solutions to~\eqref{eq:HMode}
and assume that $\epsr \in L^{\infty}(\R^2, \C^{2 \times 2})$ 
takes values in the symmetric matrices, and
that $\epsr^{-1} \in L^{\infty}(\R^2,\C^{2 \times 2})$. 
Moreover, we suppose that $\Re (\epsr^{-1})$ is pointwise 
strictly positive or strictly negative definite, and that 
$\Im(\epsr^{-1})$ is almost everywhere positive semidefinite
(even if we do never exploit the latter assumption). 
Note that we do \emph{not} assume that $\Re(\epsr^{-1})$ is 
positive definite in all of $\R^2$.

For the two-dimensional problem~\eqref{eq:HMode}, incident
electromagnetic waves reduce to $u^i(x) = \exp(\i k \, x \cdot d)
= \exp(\i k(x_1d_1 + x_2d_2))$ where $|d|=1$ and $d_2 \not = 0$.
When the incident plane wave $u^i$ illuminates the diffraction 
grating there arises a scattered field $u^s$ such that the 
total field $u=u^i+u^s$ satisfies~\eqref{eq:HMode}.
Since $\Delta u^i + k^2 u^i=0$, the scattered field satisfies
\begin{equation}
 \label{eq:HmodeEquation}
 \div( \epsr^{-1} \nabla u^s) + k^2 u^s = -\div(Q \nabla u^i)
 \quad \text{in } \R^2, \qquad \text{where } Q := \epsr^{-1} - I_2
\end{equation}
is the contrast. 
Note that $u^i$ is $\alpha$-quasi-periodic with respect $x_1$, that is, 
\[
 u^i(x_1 + 2\pi,x_2) = e^{2\pi \i \alpha} u^i(x_1,x_2)
 \qquad \text{for $\alpha: = kd_1$.}
\]
Since $u^i$ is $\alpha$-quasi-periodic and $\epsr$ is periodic,
the total field and the scattered field both are also 
$\alpha$-quasi-periodic in $x_1$. For uniqueness of solution, 
the scattered field additionally has to satisfy a radiation condition. Here 
we require that $u^s$ above (below) the dielectric structure can be
represented by a uniformly converging Fourier(-Rayleigh)
series consisting of upwards (downwards) propagating or
evanescent plane waves, see~\cite{Kirsc1993, Bonne1994},
\begin{equation}
  \label{eq:RayleighCondition}
  u^s(x) = \sum_{j \in \Z} \hat{u}^\pm_j
  e^{\i\alpha_jx_1 \pm \i\beta_j (x_2-\rho)}, \quad x_2 \gtrless \pm\rho,
  \qquad \alpha_j := j + \alpha, \quad
  \beta_j := (k^2-\alpha^2_j)^{1/2},
\end{equation}
where $\rho > \sup \{|x_2|: \, (x_1,x_2)^\top \in \supp(Q) \}$.
Let us set $\Gamma_{\pm \rho} = (-\pi,  \pi) \times \{ \pm \rho \}$. 
The numbers $\hat{u}^{\pm}_j$ are the so-called Rayleigh
coefficients of $u^s$, defined by 
\[
  \hat{u}^{\pm}_j = \frac{1}{2\pi}
  \int_{\Gamma_{\pm \rho}} u^s(x_1,x_2) \exp(-\i\alpha_j x_1) \d{s}, \qquad j \in \Z.
\]
A solution to the Helmholtz equation is called radiating   
if it satisfies~\eqref{eq:RayleighCondition}. 
 If $k^2 > \alpha^2_j$
 then the $j$th mode $\exp(i\alpha_jx_1 \pm i\beta_j (x_2\pm \rho))$
 is a propagating mode, whereas $k^2 < \alpha^2_j$ means 
 that $\exp(i\alpha_jx_1 \pm i\beta_j (x_2\pm \rho))$ is an evanescent mode.

Variational solution theory for the scattering
problem~\eqref{eq:HmodeEquation}--\eqref{eq:RayleighCondition}
is well-known, see, e.g.,~\cite{Kirsc1993, Bonne1994, Elsch1998}.
Setting 
\[
  \Omega_\rho:=(-\pi,\pi)\times(-\rho,\rho)
\]
for  $\rho > \sup \{|x_2|: \, (x_1,x_2)^\top \in \supp(Q) \}$,
one can variationally reformulate the problem in the space 
$H^1_{\alpha}(\Omega_\rho) := \{u\in H^1(\Omega_\rho): \,  u=U|_{\Omega_\rho}
  \text{ for some } \alpha\text{-quasi-periodic } U \in H^1_{\loc}(\R^2) \}$.
The resulting variational formulation is to find 
$u^s \in H^1_{\alpha}(\Omega_\rho)$ such that 
\begin{equation}
 \begin{split}
   \label{eq:variationalPeriodic}
   \int_{\Omega_\rho}(\epsr^{-1}\nabla u^s \cdot \nabla \ol{v}
    - k^2 u^s \ol{v})\d{x} 
   - \int_{\Gamma_\rho} \ol{v}T^+(u^s)\d{s}
   & - \int_{\Gamma_{-\rho}} \ol{v}T^-(u^s)\d{s} \\
   & = -\int_{\Omega_\rho} Q \nabla u^i \cdot \nabla \ol{v}\d{x}
 \end{split}
\end{equation}
for all $v \in H^1_{\alpha}(\Omega_\rho)$. The operators 
$T^{\pm}$, $\varphi \mapsto \i \sum_{j\in\Z} \beta_j \hat{\varphi}^{\pm}_je^{\i \alpha_j x_1}$,
are the so-called exterior Dirichlet-to-Neumann operators on $\Gamma_{\pm\rho}$. 
The sesquilinear form in~\eqref{eq:variationalPeriodic} is bounded
on $H^1_{\alpha}(\Omega_\rho)$ and satisfies a
G\r{a}rding inequality if, e.g., $\Re(\epsr^{-1})$ is positive definite, that is, 
$\xi^\ast \Re(\epsr^{-1})(x) \xi \geq c|\xi|^2 \geq 0$ for $\xi \in \C^2$ and 
almost every $x \in \Omega_\rho$. 
In this case, analytic Fredholm theory implies that the 
set of real wave numbers (excluding Rayleigh frequencies) where 
non-uniqueness occurs is at most countable, see~\cite{Kirsc1993, Bonne1994}. 
If $\Re (\epsr^{-1})$ changes sign, proving Fredholm properties of 
the variational formulation~\eqref{eq:variationalPeriodic} is 
non-trivial, at least if $\Im \epsr^{-1}$ vanishes. 

In this paper, we establish a Fredholm framework 
for the scattering problem via integral equation techniques, 
that is, uniqueness of solution implies existence. 
We do not aim to prove the corresponding uniqueness results, since 
for periodic scattering problems such results are anyway not 
available at all frequencies, except under restrictive geometric 
(non-trapping) conditions, see~\cite{Bonne1994}. 

\section{Integral Equation Formulation}
\label{se:integral}

In this section, we reformulate the scattering problem~\eqref{eq:HMode} 
as a volume integral equation, and prove mapping properties of the integral 
operator  between Sobolev spaces. 
To this end, let us recall that $Q = \epsr^{-1}-I_2$ and denote by 
$\ol{D} \subset \Omega_{\rho}$ the support of $Q$, restricted
to one period $\{ -\pi < x_1 < \pi \}$. 
By $G_{k,\alpha}$ we denote the Green's function to the
$\alpha$-quasi-periodic Helmholtz equation in $\R^2$, see~\cite{Kirsc1993}. 
Under the  assumption that 
\begin{equation}
  \label{eq:nonResonance}
  k^2 \neq \alpha^2_j \qquad \text{for all } j \in \Z,
\end{equation}
the $\alpha$-quasi-periodic Green's function 
has the series representation 
\begin{equation}
  \label{eq:GkAlpha}
  G_{k,\alpha}(x)
  := \frac{\i}{4\pi} \sum_{j \in \Z}
  \frac{1}{\beta_j} \exp(\i\alpha_j x_1 + \i\beta_j |x_2|), 
  \quad  x= \left( \begin{matrix} x_1 \\ x_2  \end{matrix} \right) \in \R^2, \ 
  x \not = \left( \begin{matrix} 2\pi m \\ 0 \end{matrix} \right) \text{ for } m\in\Z.
\end{equation}
Due to~\eqref{eq:nonResonance} all the 
$\beta_j=(k^2-\alpha^2_j)^{1/2}$ are non-zero. 
  
\begin{lemma}[Cf.~\cite{Kirsc1993}]
  \label{th:decomposeGreen}
  The Green's function $G_{k,\alpha}$ can be split into
  $G_{k,\alpha}(x) = (\i/4) H^{(1)}_0(k|x|) + \Psi(x)$ in
  $\R^2$ where $\Psi$ is an analytic function solving the 
  Helmholtz equation $\Delta \Psi + k^2 \Psi = 0$ in $(-2\pi,2\pi) \times \R$.
\end{lemma}

We also define a periodized Green's
function, firstly setting
\begin{equation}
  \label{eq:DefineKernel}
  \K_\rho(x) := G_{k,\alpha}(x),
  \quad x= \left( \begin{matrix} x_1 \\ x_2 \end{matrix} \right) \in \R \times (-\rho, \rho), \
  x \not = \left( \begin{matrix} 2\pi m \\  0 \end{matrix} \right) \text{ for } m\in\Z,
\end{equation}
and secondly extending $\K_\rho(x)$ $2\rho$-periodically
in $x_2$ to $\R^2$. 

The trigonometric polynomials
\begin{equation}
  \label{eq:trigonometricPolynomial}
  \varphi_j(x)
  := \frac{1}{\sqrt{4\pi\rho}}
  \exp\Big( {\i(j_1 + \alpha)x_1 + \i\frac{j_2\pi}{\rho}x_2}\Big),
  \quad j = \left( \begin{matrix} j_1 \\ j_2 \end{matrix} \right) \in \Z^2,
\end{equation}
are orthonormal in $L^2(\Omega_{\rho})$. They differ 
from the usual Fourier basis (see,
e.g.,~\cite[Section 10.5.2]{Saran2002}) only by 
a phase factor $\exp(\i \alpha x_1)$, and hence also 
form a basis of $L^2(\Omega_{\rho})$. For 
$f\in L^2(\Omega_{\rho})$ and $j =(j_1,j_2)^\top \in \Z^2$,
$\hat{f}(j) := \int_{\Omega_{\rho}} f \, \ol{\varphi_j} \d{x}$
are the Fourier coefficients of $f$.
For $0\leq s<\infty$ we define a fractional Sobolev space 
$H^s_{\per}(\Omega_{\rho})$ as the subspace of
functions in $L^2(\Omega_{\rho})$ such that
\begin{equation}
  \label{eq:sobolevSpaces}
  \| f \|^2_{H^s_{\per}(\Omega_{\rho})}
  = \sum_{j \in \Z^2}(1 + |j|^2)^s |\hat{f}(j)|^2 < \infty.
\end{equation}
It is well-known that for integer values of $s$, these 
spaces correspond to spaces of $\alpha$-quasi-periodic functions that are
$s$ times weakly differentiable, and that the above norm is then
equivalent to the usual integral norms.

Lemma~\ref{th:decomposeGreen} implies in particular that
$\K_\rho$ has an integrable singularity and that the Fourier
coefficients $\hat{\K}_\rho(j)$ are well-defined. To compute 
these coefficients explicitly, we set 
\[
  \lambda_j
  := k^2 -(j_1 + \alpha)^2 - \bigg( \frac{j_2 \pi}{\rho} \bigg)^{2}
  \quad \text{for } j\in \Z^2.
\]

\begin{theorem}
  \label{th:KHat}
  Assume that $k^2 \not = \alpha_j^2$ for all $j \in \Z$.
  Then the Fourier coefficients of the kernel $\K_\rho$
  from~\eqref{eq:DefineKernel} are given by
  \[
    \hat\K_\rho(j)
    = \begin{cases}
      \frac{\cos(j_2 \pi)e^{\i \beta_{j_1} \rho}-1}
        {\sqrt{4\pi \rho} \, \lambda_j} &  \text{for } \lambda_j \neq 0, \\
      \frac{\i}{4 j_2} \left(\frac{\rho}{\pi}\right)^{3/2} & \text{else},
    \end{cases}
    \qquad j = \left( \begin{matrix} j_1 \\ j_2 \end{matrix} \right) \in \Z^2.
  \]
\end{theorem}

\begin{remark}
Note that $\hat\K_\rho(j)$ is well-defined for $\lambda_j = 0$: Since 
$k^2 \not = \alpha_n^2$ for all $n\in \Z$, the definition of 
$\lambda_j$ implies that $j_2 \not = 0$ whenever $\lambda_j = 0$.
For completeness, we include a proof, noting that the case 
$\lambda_j \not=0$ is also shown in~\cite[Section 7.1]{Sandf2010}.
\end{remark}

\begin{proof}
 It is easy to check that $(\Delta + k^2)\varphi_j = \lambda_j \varphi_j$
 for $j=(j_1,j_2)^\top \in \Z^2$.
 If $\lambda_j \not = 0$, Green's second identity implies that
 \begin{align}
   \hat{\K}_\rho(j)
   &= \int_{\Omega_{\rho}} \K_\rho(x)\ol{\varphi_j(x)}\d{x}
   = \lambda_j^{-1} \lim_{\delta\rightarrow 0}\int_{\Omega_{\rho}\setminus B(0,\delta)}
     G_{k,\alpha}(x)\ol{(\Delta + k^2)\varphi_j(x)} \d{x} \nonumber \\
   & = \lambda_j^{-1} \lim_{\delta\rightarrow 0}\left[\left(\int_{\partial \Omega_{\rho}}
     + \int_{\partial B(0,\delta)}\right) \left(G_{k,\alpha}
     \frac{\partial\ol{\varphi_j}}{\partial \nu}
     - \frac{\partial G_{k,\alpha}}{\partial\nu}\ol{\varphi_j}\right)\d{s}\right.
     \label{eq:2} \\
   & \hspace*{4cm} + \left.  \int_{\Omega_{\rho}\setminus B(0,\delta)} (\Delta + k^2) G_{k,\alpha}(x)\ol{\varphi_j(x)} \d{x} \right],
 \end{align}
 where $\nu$ denotes the exterior normal vector to $B(0,\delta)$.
 The last volume integral vanishes since
 $(\Delta + k^2) G_{k,\alpha} = 0$ in
 $\Omega_{\rho}\setminus B(0,\delta)$ for any $\delta>0$.
 Let us now consider the first integral in~\eqref{eq:2}.
 The boundary of $\Omega_{\rho}$ consists of two
 horizontal lines $\Gamma_{\pm \rho}$ and two vertical
 lines $\{ (x_1,x_2): \, x_1 = \pm \pi, \, -\rho < x_2 < \rho \}$.
 Hence, the normal vector $\nu$ on these boundaries is either
 $(\pm 1, 0)^\top$ or $(0, \pm 1)^\top$. Straightforward
 computations yield that
 \begin{eqnarray}
    G_{k,\alpha}(x_1,\pm \rho)
   = \frac{\i}{4\pi}\sum_{n\in \Z}\frac{e^{\i\beta_n \rho}}{\beta_n} e^{\i \alpha_n x_1},\quad 
    \partial_2 G_{k,\alpha}(x_1,\pm \rho)
   = \mp \frac{1}{4\pi}\sum_{n\in \Z} e^{\i\beta_n \rho} e^{\i \alpha_n x_1}, \label{eq:3} \\
    \ol{\varphi_j(x_1,\pm \rho)}
   = \frac{1}{\sqrt{4 \pi \rho}} e^{-\i\alpha_{j_1}x_1} \cos(j_2 \pi),
    \quad \text{and} \quad \partial_2\ol{\varphi_j(x_1,\pm \rho)}
   = -\frac{\i j_2 \pi}{\rho} \ol{\varphi(x_1,\pm \rho)}. \label{eq:4}
 \end{eqnarray}
 In consequence,
 \begin{align*}
   \int_{\Gamma_{\pm \rho}} \left(G_{k,\alpha} \frac{\partial\ol{\varphi_j}}{\partial \nu} - \frac{\partial G_{k,\alpha}}{\partial \nu} \ol{\varphi_j} \right) \d{s}
   &= -\int_{\Gamma_{\rho}} \partial_2 G_{k,\alpha} \ol{\varphi_j} \d{s}
   + \int_{\Gamma_{-\rho}} \partial_2 G_{k,\alpha} \ol{\varphi_j}\d{s} \\
   &= -2\int_{\Gamma_{\rho}} \partial_2 G_{k,\alpha} \ol{\varphi_j} \d{s}.
 \end{align*}
 Using the above formulas for $\partial_2 G_{k,\alpha}$ and
 $\ol{\varphi_j}$ in~\eqref{eq:3} and~\eqref{eq:4},
 respectively, we find that
 \[
   -2 \int_{\Gamma_{\rho}} \partial_2 G_{k,\alpha}\ol{\varphi_j} \d{s}
   = \frac{\cos(j_2 \pi)}{\sqrt{4\pi \rho}} e^{\i \beta_{j_1} \rho}.
 \]
 Computing the partial derivatives of $G_{k,\alpha}$
 and $\varphi_j$ with respect to $x_1$ analogously to the 
 above computations, one finds that the integrals on the
 vertical boundaries of $\Omega_{\rho}$ vanish due to 
 the $\alpha$-quasi-periodicity of both functions. 
 Thus, we obtain that
 \begin{equation}
 \label{eq:boundary2h}
   \int_{\partial \Omega_{\rho}}
   \left(G_{k,\alpha}\frac{\partial\ol{\varphi_j}}{\partial \nu}
   -\frac{\partial G_{k,\alpha}}{\partial\nu} \ol{\varphi_j} \right)\d{s}
   = \frac{\cos(j_2 \pi)}{\sqrt{4\pi \rho}} e^{\i \beta_{j_1} \rho}.
 \end{equation}
 Now we consider the second integral in~\eqref{eq:2}.
 From Lemma~\ref{th:decomposeGreen} we know that
 $G_{k,\alpha}(x) = \frac{\i}{4} H^{(1)}_0(k|x|) + \Psi(x)$
 where $\Psi$ is a smooth function in $\Omega_\rho$. Obviously,
 \[
  \lim_{\delta\rightarrow 0} \int_{\partial B(0,\delta)}\left(\Psi\frac{\partial\ol{\varphi_j}}{\partial \nu} - \frac{\partial \Psi}{\partial \nu}\ol{\varphi_j}\right) \d{s} = 0.
 \]
 The asymptotics of $H^{(1)}_0$ and its derivative for small arguments,
 \[
   H^{(1)}_0(r)
   = \frac{2\i}{\pi}\log r + \mathcal{O}(1)
   \quad \text{and} \quad 
   (H^{(1)}_0)'(r)
   = \frac{2\i}{\pi r} + \mathcal{O}(1) \quad \text{as } r \to 0,
 \]
 allow to show that
 \begin{equation}
   \label{eq:boundaryDelta}
   \lim_{\delta\rightarrow 0} \int_{\partial B(0,\delta)}
   \left(G_{k,\alpha}\frac{\partial\ol{\varphi_j}}{\partial\nu}
   -\frac{\partial G_{k,\alpha}}{\partial r}\ol{\varphi_j}\right) \d{s},
   = - \frac{1}{\sqrt{4 \pi \rho}},
 \end{equation}
 see, e.g.,~\cite[Theorem 2.2.1]{Saran2002}.
 Combining~\eqref{eq:boundary2h}
 with~\eqref{eq:boundaryDelta} yields that
 \[
   \K_\rho(j) = 
   \frac{1}{\sqrt{4\pi\rho} \lambda_j}(\cos(j_2 \pi) e^{\i\beta_{j_1}\rho}-1)
   \quad \text{for } \lambda_j \not = 0.
 \]
 For $\lambda_j = 0$ we use de
 L'H\^{o}spital's rule to find that
 \[
   \K_\rho(j) 
   = \lim_{\gamma \to (j_1+\alpha)^2 + (j_2 \pi / \rho)^2}
     \frac{\cos(j_2 \pi) \exp(\i \rho \sqrt{\gamma - (j_1+\alpha)^2})-1}
     {\sqrt{4 \pi \rho} \, [\gamma - (j_1+\alpha)^2 - (j_2 \pi / \rho)^2]} 
   = \frac{\i \rho^{3/2}}{4 \pi^{3/2} j_2}.
 \]
Note that the assumption that
$k^2 \not = \alpha_j^2$ for all $j\in\Z^2$ implies
that $\lambda_j$ and $j_2$ cannot vanish simultaneously.
 \end{proof}

Since the Fourier coefficients of the $\K_\rho$
decay quadratically, 
$|\hat{\K}_\rho(j)| \leq C / (1 + (j_1 +\alpha)^2 + (j_2 \pi /\rho)^2)$
for $j \in \Z^2$, the convolution operator with kernel $\K_\rho$ is 
bounded from $L^2(\Omega_{\rho})$ into $H^2_{\per}(\Omega_{\rho})$.

\begin{proposition}
  \label{th:boundKh}
  Assume that $k^2 \not = \alpha_j^2$ for all $j \in \Z$.
  Then the convolution operator $K_\rho$, defined by 
  \begin{equation*}
    (K_\rho f)(x) = \int_{\Omega_{\rho}} \K_\rho(x-y) f(y) \d{y}
    \quad \text{for } x\in \Omega_{\rho},
  \end{equation*}
  is bounded from $L^2(\Omega_{\rho})$ into
  $H^2_{\per}(\Omega_{\rho})$.
\end{proposition}

Recall that $D \subset \Omega_\rho$ is the
support of contrast $Q$. Let us additionally 
introduce 
\[
  \Omega := (-\pi,\pi)\times\R
\] 
and $H^\ell_{\alpha}(\Omega_\RR) := \{u\in H^\ell(\Omega_\RR): \,  u=U|_{\Omega_\RR}
  \text{ for some } \alpha\text{-quasi-periodic } U \in H^\ell_{\loc}(\R^2) \}$ for $\ell \in \N$, $R>0$.

\begin{lemma}
  \label{th:H2smoothness}
  Assume that $k^2 \not = \alpha_j^2$ for all $j \in \Z$.
  Then the volume potential $V_k$ defined by
  \[
    (V_k f)(x) = \int_{D} G_{k,\alpha}(x-y) f(y) \d{y},
    \quad x \in \Omega,
  \]
  is bounded from $L^2(D)$ into $H^2_\alpha(\Omega_\RR)$
  for all $\RR>0$.
\end{lemma}
\begin{proof}
  Consider $\chi \in C^\infty(\Omega)$ such that
  $\chi = 1$ in $D$, $0 \leq \chi \leq 1$ in
  $\Omega_\rho \setminus \ol{D}$ and $\chi(x) = 0$ for
  $|x_2|>\rho$. Then
  $V_k g = \chi V_k g + (1-\chi)V_k g$.
  Note that $(1-\chi)V_k g =
    \int_{D} (1-\chi) G(\cdot - y) g(y)\d{y}$
  is an integral operator with a smooth kernel,
  since the series in~\eqref{eq:GkAlpha} converges
  absolutely and uniformly for $|x_2| \geq \rho >0$,
  as well as all its partial derivatives. In 
  consequence, the integral operator $(1-\chi)V_k$ is
  bounded from $L^2(D)$ into $H^2_\alpha(\Omega_{\RR})$, 
  since
  \[
   \| \partial_1^{\beta_1} \partial_2^{\beta_2}
     ((1-\chi)V_k g)\|^2_{L^2(\Omega_{\RR})}
   \leq \int_{\Omega_{\RR}} \int_{D}
     | \partial_1^{\beta_1} \partial_2^{\beta_2}
     [(1-\chi(x))G_{k,\alpha}(x-y)]|^2 \d{y}\d{x} \
     \|g\|^2_{L^2(D)} 
  \]
  for all $\beta_{1,2} \in \N$ such that
  $\beta_1 + \beta_2 \leq 2$.
  It remains to show the boundedness
  of $\chi V_k$ from $L^2(D)$ into
  $H^2(\Omega_{\rho})$. Let $g \in L^2(D)$
  and consider the operator $K_{2\rho}$
  from Proposition~\ref{th:boundKh},
  mapping $L^2(\Omega_{2\rho})$ into
  $H^2_{\per}(\Omega_{2\rho}) \subset H^2_{\alpha}(\Omega_{2\rho})$,
  \[
    (K_{2\rho} g)(x) = \int_{D} \K_{2\rho}(x-y)g(y)\d{y}
    \quad \text{for } x \in \Omega_{2\rho}.
  \]
  If $x \in \Omega_\rho$, then $|x_2-y_2| \leq 2\rho$,
  that is, $\K_{2\rho}(x-y) = G_{k,\alpha}(x-y)$. Hence, 
  $K_{2\rho} g = V_k g$ in $\Omega_\rho$, and hence
  $\chi K_{2\rho} g = \chi V_k g$ in $\Omega_{\rho}$.
  Since $\chi$ is a smooth function, we conclude
  that $\chi V_k$ is bounded from $L^2(D)$
  into $H^2_\alpha(\Omega_{\rho})$.
\end{proof}

Note that the potential $V_k f$ can be extended to
an $\alpha$-quasi-periodic function in $H^2_{\loc}(\R^2)$,
due to the $\alpha$-quasi-periodicity of the kernel.

\begin{lemma}
  \label{th:GreenSolution}
  For $g \in L^2(D,\C^2)$ the potential
  $w = \div V_k g$ belongs to $H^1_\alpha(\Omega_\rho)$
  for all $\rho>0$. It is the unique
  radiating weak solution to $\Delta w + k^2 w = -\div g$
  in $\Omega$, that is, it satisfies 
  \begin{equation}
    \label{eq:pde}
    \int_{\Omega} (\nabla w\cdot\nabla \ol{v} - k^2 w\ol{v})\d{x}
    = -\int_{D} g\cdot\nabla\ol{v} \d{x}
  \end{equation}
  for all $v \in H^1_\alpha(\Omega)$ with compact
  support, and additionally the 
  Rayleigh expansion condition~\eqref{eq:RayleighCondition}.
\end{lemma}
\begin{proof}
  Lemma~\ref{th:H2smoothness} and $\alpha$-quasi-periodicity of
  the kernel of $V_k$ imply that $w$ is a function in
  $H^1_\alpha(\Omega_\rho)$ for all $\rho>0$.
  It is sufficient to prove~\eqref{eq:pde} 
  for all smooth $\alpha$-quasi-periodic test functions
  $v$ that are supported in $\{ |x_2| < C \}$ for 
  some $C>0$ depending on $v$. 
  It is well-known that $p=V_kg \in H^2_\alpha(\Omega)$ 
  is a weak solution to the Helmholtz equation, that is, 
  \[
    \int_{\Omega} (\nabla p_j \cdot \nabla \partial_j \ol{v} - k^2 p_j \partial_j \ol{v})\d{x}
    = -\int_{D} g_j \partial_j\ol{v} \d{x}
  \]
  for $j=1,2$. An integration by parts shows that 
  \[
    \int_{\Omega} (\nabla \div p \, \cdot \nabla \ol{v} - k^2 \div p \, \ol{v})\d{x}
    = -\int_{D} g \cdot \nabla \ol{v} \d{x},
  \]
  which implies~\eqref{eq:pde} due to $\div p = w$. 
  Since the components of the potential $p=V_kg$ satisfy 
  the Rayleigh condition, a simple computation shows that 
  the divergence $w = \div p$ does also satisfy the latter condition.
  Uniqueness of a radiating solution to~\eqref{eq:pde} when $g=0$ 
  follows from the Rayleigh expansion condition and a unique 
  continuation argument. 
%
\end{proof}

Returning to the differential equation~\eqref{eq:HmodeEquation}
for the scattered field $u^s$, let us set $f=Q\nabla u^i$. (Recall that 
$Q = \epsr^{-1} - I_2$.) The variational
formulation of~\eqref{eq:HmodeEquation} is
\begin{equation}
  \label{eq:variationalForm}
  \int_{\Omega} (\nabla u^s \cdot \nabla\ol{v} - k^2 u^s \ol{v}) \d{x}
  = -\int_{D} (Q \nabla u^s + f) \cdot \nabla \ol{v} \d{x}
\end{equation}
for all $v \in H^1_\alpha(\Omega)$ with compact
support in $\ol{\Omega}$. 
From Lemma~\ref{th:GreenSolution} we know that
the radiating solution to this problem is given by
$u^s = \div V_k(Q\nabla u^s + f)$. Hence, we aim to find 
$u^s: \, \Omega \to \C$ that belongs to $H^1_{\alpha}(\Omega_{\RR})$
for all $\RR>0$, such that 
\begin{equation}
  \label{eq:lippmann}
  u^s - \div V_k(Q \nabla u^s) = \div V_k(f) \qquad \text{in } \Omega.
\end{equation}

\section{G\r{a}rding Inequalities in Weighted Sobolev Spaces}
\label{se:anisotropic}

For scattering problems in free space and for scalar and positive contrast, 
the paper~\cite{Kirsc2009} investigates integral equations similar 
to~\eqref{eq:lippmann} in weighted spaces. In this section we generalize 
the results from~\cite{Kirsc2009} to anisotropic and possibly sign-changing 
coefficients in a periodic setting, proving a G\r{a}rding 
inequality for $I-\div V_k(Q\nabla\cdot)$ in a anisotropically weighted $\alpha$-quasi-periodic
$H^1$-space.

From~\eqref{eq:lippmann} it is obvious that the knowledge of 
$u$ in $D$ is sufficient to determine $u$ in $\Omega \setminus \ol{D}$ 
by integration. Thus, we define the operator $L_k: \, f \mapsto \div V_k f$ 
that is bounded from $L^2(D,\C^2)$ into $H^1_\alpha(D)$ and consider the 
integral equation
\begin{equation}
  \label{eq:lippmannD}
  u = L_k(Q \nabla u + f) \quad \text{in } H^1_\alpha(D).
\end{equation}
To study G\r{a}rding inequalities for volume
integral equations, we introduce suitable weighted
Sobolev spaces. 
To this end, we recall that the symmetric $2 \times 2$ matrix $\Re(Q)$ 
has pointwise almost everywhere in $D$ an eigenvalue decomposition $\Re(Q) = U^\ast \Sigma U$ 
with a diagonal matrix $\Sigma$ and an orthogonal matrix $U$.  This decomposition 
can be used to define the absolute value $| \Re(Q) | = U^\ast |\Sigma| U$ and the 
square root $| \Re(Q) |^{1/2} = U |\Sigma|^{1/2} U^\ast$, where the absolute value 
and the square root are element-wise applied to the diagonal matrix $\Sigma$. 
The two eigenvalues $\lambda_{1,2}$ of $\Re(Q)$ define
\begin{equation}
  \label{eq:lambdaMinMax}
  \lambda_{\min}(x) = \min \{ |\lambda_{1}(x)|, \, |\lambda_{2}(x)| \}, \quad 
  \lambda_{\max}(x) = \max \{ |\lambda_{1}(x)|, \, |\lambda_{2}(x)| \}, \quad x \in D.
\end{equation}
We assume in the following that $\Re(Q)$ is pointwise either strictly positive or strictly negative 
definite, such that we can assign a sign function $\sign(\Re(Q)) \in L^\infty(\Omega)$ to $\Re(Q)$, 
indicating whether the eigenvalues of $\Re(Q)$ are positive or negative. In the sequel, we write 
$\Re(Q) \geq c$ in $D$ ($\Re(Q) \leq c$ in $D$) to indicate that the eigenvalues 
$\lambda_{1,2}$ are larger than or equal to (less than or equal to) a constant $c$, 
almost everywhere in $D$. Note that the spectral matrix norm is denoted by $|\cdot |_2$.

We denote by $H^1_{\alpha,Q}(D)$ the
completion of $H^1_\alpha(D)$ with respect to the norm
$\Vert \cdot \Vert_{H^1_{\alpha,Q}(D)}$,
\begin{equation}
  \label{eq:weightedNorm}
  \Vert u \Vert_{H^1_{\alpha,Q}(D)}^2 := \| \sqrt{|\Re(Q)|} \nabla u \|_{L^2(D,\C^2)}^2 + \| u \|_{L^2(D)}^2.
\end{equation}
Since we assumed that $\supp (\Re(Q)) = \ol{D}$, this 
norm is non-degenerate. 
Moreover, $\Vert u \Vert_{H^1_{\alpha,Q}(D)}$ is an equivalent
norm in $H^1_\alpha(D)$ provided that $|\Re(Q)|$ is bounded 
from below in $D$ by some positive constant. In general, 
$\Vert u \Vert_{H^1_{\alpha,Q}(D)} 
  \leq (1+\| | \sqrt{|\Re(Q)|} |_2 \|_{L^\infty(D)}) \, \Vert u \Vert_{H^1_\alpha(D)}$.
Note also that the norm of $H^1_{\alpha,Q}(D)$ is linked to the sesquilinear form
\begin{equation}
  \label{eq:aq}
  a_Q(u,v) = \int_D \big[ \sign(\Re(Q)) Q \nabla u \cdot \nabla \ol{v} + u\ol{v} \big]\d{x},
 \qquad u,v \in H^1_{\alpha,Q}(D).
\end{equation}
Indeed, $\Vert u \Vert_{H^1_{\alpha,Q}(D)}^2 = \Re \left[ a_Q(u,u) \right]$
for $u \in H^1_{\alpha,Q}(D)$.  In consequence, 
the form $a_Q$ is non-degenerate, that is, if $a_Q(u,v) = 0$ for all 
$v \in H^1_{\alpha,Q}(D)$, then $u = 0$.

If $\Im Q$ vanishes in $D$ (that is, the values of 
$x \mapsto Q(x)$ are self-adjoint matrices), then $a_Q$ is simply the inner 
product associated with the norm of $H^1_{\alpha,Q}(D)$,
\[
  \left\langle u, \, v \right\rangle_{H^1_{\alpha,Q}(D)}
  = \int_D \big[ |Q| \nabla u \cdot \nabla \ol{v} + u\ol{v}\big]\d{x},
  \qquad u,v \in H^1_{\alpha,Q}(D).
\]

\begin{lemma}
  \label{th:bounded}
  Assume that there exists $C>0$ such that 
  \begin{equation}
    \label{eq;boundIm}
    |\Im (Q(x)) \xi | \leq C |\Re (Q(x)) \xi | \quad \text{ for almost every 
  $x \in D$ and all $\xi \in \C^2$.}
  \end{equation}
  Then $v \mapsto L_k(Q\nabla v)$ is bounded on $H^1_{\alpha,Q}(D)$.
\end{lemma}
\begin{proof}
  Due to Theorem~\ref{th:H2smoothness}, $L_k$ is bounded from
  $L^2(D,\C^2)$ into $H^1_\alpha(D)$. Furthermore,  $v\mapsto Q\nabla v$
  is bounded from $H^1_{\alpha,Q}(D)$ into $L^2(D,\C^2)$, since
  \begin{equation}
    \label{eq:imbedding2}
    \begin{split}
    \| Q \nabla u \|_{L^2(D, \C^2)}
    & \leq \| \Re (Q)  \nabla u \|_{L^2(D, \C^2)} + \| \Im (Q)  \nabla u \|_{L^2(D, \C^2)} \\
    & \leq \| |\Re (Q)|  \nabla u \|_{L^2(D, \C^2)} + C \| \Re (Q)  \nabla u \|_{L^2(D, \C^2)}\\
    & \leq (1+C)\| |\sqrt{| \Re(Q) |} |_2 \|_{L^\infty(D)} \| u \|_{H^1_{\alpha,Q}(D)}.
    \end{split}    
  \end{equation}
  Moreover, the imbedding $H^1_\alpha(D)\subset H^1_{\alpha,Q}(D)$ is
  bounded, as mentioned above.
  Hence, $v \mapsto L_k(Q\nabla v)$ is bounded on $H^1_{\alpha,Q}(D)$.
\end{proof}

\begin{remark}
  Condition~\eqref{eq;boundIm} is satisfied if the absolute values of the 
  eigenvalues of $\Im Q$ are pointwise bounded by $C \lambda_{\min}$
  (recall from~\eqref{eq:lambdaMinMax} that $\lambda_{\min}$ is the 
  minimum of the absolute values of the eigenvalues of $\Re(Q)$).
\end{remark}

If $u \in H^1_\alpha(D) \subset H^1_{\alpha,Q}(D)$ solves the
Lippmann-Schwinger equation~\eqref{eq:lippmannD},
then Lemma~\eqref{th:bounded} implies that $u$ solves
the same equation in $H^1_{\alpha,Q}(D)$. Since $a_Q$ is
non-degenerate, solving the Lippmann-Schwinger equation
in $H^1_{\alpha,Q}(D)$ is equivalent to solve
\begin{equation}
  \label{eq:lippmannVariational}
  a_Q(u - L_k(Q \nabla u + f), \, v) = 0
  \qquad \text{for all } v \in H^1_{\alpha,Q}(D).
\end{equation}
If $u \in H^1_{\alpha,Q}(D)$ solves the latter variational problem
for some $f \in L^2(D,\C^2)$, then $u = L_k(Q \nabla u+f)$
belongs to $H^1_\alpha(D)$, due to~\eqref{eq:imbedding2} and
since $L_k$ is bounded from $L^2(D, \C^2)$ into $H^1_\alpha(D)$.

\begin{proposition}
  \label{th:equivalenceWeight-NoWeight}
  Assume that $f \in L^2(D,\C^2)$.
  Then any solution to the Lippmann-Schwinger
  equation~\eqref{eq:lippmannD} in $H^1_\alpha(D)$ is
  a solution in $H^1_{\alpha,Q}(D)$ and vice versa.
\end{proposition}

Our aim is now to prove a (generalized) G\r{a}rding
inequality for the variational
problem~\eqref{eq:lippmannVariational}. 
The following lemma will turn out to be useful.

\begin{lemma}
\label{th:Compact1}
Suppose that $X$ and $Y$ are Hilbert spaces. Let $T_{1,2}$ 
be bounded linear operators from $X$ into $Y$ and
consider the sesquilinear form $a: \, X\times X \rightarrow \C$, 
defined by $a(u,v) = \langle T_1 u, T_2 v\rangle_Y$ for $u,v \in X$.
If one of the operators $T_1$ and $T_2$ is compact, then
the linear operator $A:\, X\rightarrow X$, defined by 
$\langle Au,v\rangle_X = a(u,v)$ for all $u,v \in X$, is compact, too.
\end{lemma}
\begin{proof}
It is easily seen that $A$ is a well-defined bounded linear operator.
Obviously,
$|\langle Au,v\rangle_X| = |a(u,v)| \leq C\|T_1 u\|_Y \|T_2 v\|_Y$
for $u,v \in X$. Assume that $T_1$ is compact, and note that  
\[
  \|A u\|_X = \sup_{0\neq v\in X}\frac{|\langle Au,v\rangle_{X}|}{\|v\|_{X}}\leq C\|T_1 u\|_Y.
\]
If a sequence $\{ u_n \}$ converges weakly to zero in $X$,
then $\{ T_1 u_n \}$ contains a strongly convergent subsequence 
tending to zero in $Y$. 
Consequently, $\{ A u_n \}$ also contains a strongly convergent zero 
sequence, which means that $A$ is compact. One can 
analogously derive the compactness of $T$ in case that $T_2$ 
is compact, since $a(u,v) = \langle T_2^\ast T_1 u,v\rangle$.
\end{proof}

The next lemma proves G\r{a}rding inequalities for the
operator $v \mapsto v-L_k(Q \nabla v)$ using the
sesquilinear form $a_Q$ from~\eqref{eq:aq}.
The second part of the claim uses a periodic extension operator
\[
  E: \, H^1_{\alpha}(D) \to H^1_{\alpha}(\Omega),
  \quad \left. E(u) \right|_{D} = u, 
  \quad \left. E(u) \right|_{\Omega \setminus \Omega_{2\rho}} = 0, 
\]
introduced in Appendix~\ref{se:extension}. 
The operator norm of $E$ is
\[
  \| E \|_{H^1_\alpha(D) \to H^1_\alpha(\Omega_{2\rho})}
  = \left(1 + \| E \|_{H^1_\alpha(D) \to H^1_\alpha(\Omega_{2\rho}\setminus \ol{D})}^2 \right)^{1/2}.
\]

\begin{theorem}
  \label{l-coerc_2}
  Assume that $D$ is a Lipschitz domain and that $Q \in L^\infty(D, \C^{2 \times 2})$. 

  (a) If $\Re(Q) > 0$ in $D$, then there
  exists a compact operator $K_+$ on $H^1_{\alpha,Q}(D)$ such that
  \begin{equation}
    \label{eq:coerc_2}
    \Re \left[ a_Q( v-L_k(Q\nabla v), \, v ) \right] \geq
    \Vert v\Vert^2_{H^1_{\alpha,Q}(D)} - \Re\langle K_+v, \, v\rangle_{H^1_{\alpha,Q}(D)},
    \qquad v\in H^1_{\alpha,Q}(D).
  \end{equation}

  (b) If $\Re(Q)<-1$, and if
  \begin{equation}
    \label{eq:condition}
    \| E \|_{H^1_{\alpha}(D) \to H^1_{\alpha}(\Omega_{2\rho})} 
    <  \inf_D | \Re(Q) |^{1/2}_2, 
  \end{equation}
  then there exists a constant $C>0$ and a 
  compact operator $K_-$ on $H^1_{\alpha,Q}(D)$ such that
  \begin{equation}
    \label{eq:coerc_3}
    - \Re \left[ a_Q( v-L_k(Q\nabla v), \, v ) \right] 
    \geq C \Vert v\Vert^2_{H^1_{\alpha,Q}(D)}
    - \Re\langle K_-v, \, v\rangle_{H^1_{\alpha,Q}(D)},
    \qquad v\in H^1_{\alpha,Q}(D). 
  \end{equation}
\end{theorem}

\begin{remark}
  \label{remark1}
  If $\Im(Q) = 0$ in $D$, then both statements~\eqref{eq:coerc_2}
  and~\eqref{eq:coerc_3} are nothing but standard G\r{a}rding
  estimates: The form $a_Q$ defines an inner product on
  $H^1_{\alpha,Q}(D)$, and, e.g.,~\eqref{eq:coerc_2} can be
  rewritten as 
  $\Re \bigl\langle v-L_k(Q\nabla v), \, v \bigr\rangle
    \geq \Vert v\Vert^2 - \Re\langle K_+v, \, v\rangle$
  for $v\in H^1_{\alpha,Q}(D)$.
\end{remark}

\begin{proof}
(a) We start with the case $\Re(Q)>0$ in $D$. 
Let $v\in H^1_{\alpha,Q}(D)$ and define $w$ by
\begin{equation}
  \label{eq:defineW}
  w = L_\i(Q\nabla v)
  = \div \int_D G_{\i,\alpha}(\cdot-y) [Q(y)\,\nabla v(y)] \d{y}
  \quad \text{in }\Omega.
\end{equation}
Then $w\in H_{\alpha}^1(\Omega)$ decays exponentially to zero
as $|x_2|$ tends to infinity. Moreover, $\Delta w-w=-\div (Q\nabla v)$
holds in $\Omega$ in the weak sense due to Lemma~\ref{th:GreenSolution}, 
that is,
\begin{equation}
 \label{eq:variationalW}
  \int_{\Omega}\bigl[ \nabla w \cdot \nabla \ol{\psi}+ w \overline{\psi}  \bigr]\d{x}
  = - \int_D Q \nabla v \cdot \nabla \ol{\psi} \d{x} \quad\mbox{for all }
  \psi\in H_{\alpha}^1(\Omega).
\end{equation}
Setting $\psi=w$, we find that
$-\Re \int_D Q \nabla v \cdot \nabla \ol{w} \d{x} = \| w \|^2_{H^1(\Omega)}$.
Hence, 
\begin{align*}
  \Re \left[ a_Q( v-L_\i(Q \nabla v), \, v ) \right]
  & = \int_D \bigl[ \Re (Q) \nabla v \cdot \nabla \ol{v} + |v|^2 \bigr]  \d{x}
       - \Re \int_D \bigl[ Q \nabla w \cdot \nabla \ol{v}  + w \overline{v}\bigr]\d{x} \\
  & = \int_D\bigl[  | \sqrt{\Re(Q)} \nabla v|^2+|v|^2 - \Re(w\overline{v})\bigr]\d{x}
    +\int_{\Omega} \bigl[|\nabla w|^2+|w|^2\bigr]\d{x} \\
  & \geq \Vert v\Vert^2_{H^1_{\alpha,Q}(D)} - \frac{1}{2}\|v\|_{L^2(D)}^2 +\ \frac{1}{2}
    \int_D \bigl[|v|^2+|w|^2-2\Re(w\overline{v}) \bigr] \d{x},    
\end{align*}
where the last term on the right is positive. In consequence, 
\begin{align*}
 \Re \left[ a_Q( v-L_k(Q\nabla v), \, v ) \right]
 \geq \Vert v\Vert^2_{H^1_{\alpha,Q}(D)} - \frac{1}{2}\langle v, \, v\rangle_{L^2(D)}
 - \Re\left[ a_Q( (L_k-L_\i)(Q\nabla v), \, v ) \right]
\end{align*}
for all $v\in H^1_{\alpha,Q}(D)$.
Due to Lemma~\ref{th:Compact1} and Rellich's lemma there exists 
a compact operator $K_1$ on $H^1_{\alpha,Q}(D)$ such that
$\langle v,v\rangle_{L^2(D)}
  = 2 \Re\langle K_1v, \, v\rangle_{H^1_{\alpha,Q}(D)}$.
Further, the operator $(L_k - L_\i)(Q\nabla \cdot)$ is compact on
$H^1_\alpha(D)$ due to the smoothness of the kernel
shown in Appendix~\ref{se:appendixA}. Hence the operator 
$K_2$ defined by $\langle K_2v, \, v\rangle_{H^1_{\alpha,Q}(D)} = a_Q( (L_k-L_\i)(Q\nabla v), \, v )$
is compact on $H^1_{\alpha,Q}(D)$ due to Lemma~\ref{th:Compact1} 
and the boundedness of the imbedding
$H^1_\alpha(D)\subset H^1_{\alpha,Q}(D)$.
Setting $K_+:=K_1+K_2$, we obtain the claimed generalized
G\r{a}rding inequality.

(b) Now we consider the case that $\Re(Q) < -1$ in $D$, and 
assume additionally that~\eqref{eq:condition} holds.
As in the first part of the proof, the variational
formulation~\eqref{eq:variationalW} for $w$, defined as
in~\eqref{eq:defineW}, yields that
\begin{align*}
  - \Re \left[ a_Q( v-L_\i(Q\nabla v), \, v ) \right]
  & = \Re\int_D \bigl[ \Re (Q) \nabla v \cdot \nabla \ol{v} - |v|^2
       - Q \nabla w \cdot \nabla \ol{v} + w \overline{v}\bigr]\d{x} \\
  & = - \int_D\bigl[ |\sqrt{|\Re (Q)|} \nabla v|^2 + |v|^2 \bigr]\d{x}
       + \| w \|^2_{H^1_\alpha(\Omega)} 
       + \Re \int_D w\overline{v} \d{x} \\
  & \geq \| w \|^2_{H^1_\alpha(\Omega)}
       - \| v \|_{H^1_{\alpha,Q}(D)}^2 + \Re \int_D w\overline{v} \d{x}.
\end{align*}
We plug in $\psi = -E(v)$ into~\eqref{eq:variationalW}
and take the real part of that equation, to find that
\begin{align*}
  \| \sqrt{|\Re(Q)|} \nabla v \|_{L^2(D, \C^2)}^2
  & \leq \| w \|_{H^1_\alpha(\Omega)} \| E(v) \|_{H^1_\alpha(\Omega)} 
  \leq \| E \|_{H^1_\alpha(D) \to H^1_\alpha(\Omega_{2\rho})} \, \| w \|_{H^1_{\alpha}(\Omega)} \| v \|_{H^1_{\alpha}(D)} \\
  & \leq \| E \| \, \| w \|_{H^1_{\alpha}(\Omega)} \left( \| | \sqrt{|\Re(Q)|}^{-1} |_2 \|_{L^\infty(D)} \| v \|_{H^1_{\alpha,Q}(D)}  + \| v \|_{L^2(D)} \right).
\end{align*}
For $x \in D$, the spectral matrix norm $| \sqrt{|\Re(Q)|}^{-1}(x) |_2$ of the inverse of 
$\sqrt{|\Re(Q)|}(x)$ equals the reciprocal value $\lambda_{\min}(x)^{-1/2}$ 
($\lambda_{\min,\max}$ are the smallest/largest eigenvalue, in magnitude, of $\Re (Q)$, 
see~\eqref{eq:lambdaMinMax}). Note that  
\begin{align*}
  \| | \sqrt{|\Re(Q)|}^{-1} |_2 \|_{L^\infty(D)}^{-1} 
  & = [ \sup_{x \in D} \lambda_{\min}(x)^{-1/2} ]^{-1}
  = \inf_{x \in D} \lambda_{\min}(x)^{1/2} 
  \leq \sup_{x \in D} \lambda_{\max}(x)^{1/2} \\
  & \leq [1+ \sup_{x \in D} \lambda_{\max}(x)]^{1/2}
  = [1+ \| | \Re (Q) |_2 \|_{L^\infty(D)} ]^{1/2}. 
\end{align*}
Next, we estimate that 
\begin{multline*}
  \| v \|_{H^1_{\alpha,Q}(D)}^2 - \big[1+\| |\Re(Q)|_2 \|_{L^\infty(D)} \big] \| v \|_{L^2(D)}^2
  \leq \| v \|_{H^1_{\alpha,Q}(D)}^2 - \| v \|_{L^2(D)}^2 \\
  \leq \| E \| \,  \| | \sqrt{|\Re(Q)|}^{-1} |_2 \|_{L^\infty(D)} \, \| w \|_{H^1_{\alpha}(\Omega)} \\
  \left( \| v \|_{H^1_{\alpha,Q}(D)}  +  [1+ \| | \Re (Q) |_2 \|_{L^\infty(D)} ]^{1/2}  \| v \|_{L^2(D)} \right).
\end{multline*}
Dividing by the term in brackets on the right, we obtain that
\begin{equation}
  \label{eq:intermediateResult}
  \| v \|_{H^1_{\alpha,Q}(D)} - \big[1+\| |\Re(Q)|_2 \|_\infty \big]^{1/2} \| v \|_{L^2(D)}
  \leq \| E \| \,  \| | \sqrt{|\Re(Q)|}^{-1} |_2 \|_{L^\infty(D)}  \, \| w \|_{H^1_{\alpha}(\Omega)}.
\end{equation}
Note that the constant
\[
  c := \| E \|_{H^1_{\alpha}(D) \to H^1_{\alpha}(\Omega_{2\rho})} \
  \| | \sqrt{|\Re(Q)|}^{-1} |_2 \|_{L^\infty(D)}
\]
is by assumption~\eqref{eq:condition} less than one. If we set for a moment,
$C=[1+\| | \Re(Q) |_2 \|_\infty]^{1/2}$ then~\eqref{eq:intermediateResult}
and Cauchy's inequality imply that
\begin{align*}
  c^2 \| w \|_{H^1_{\alpha}(\Omega)}^2
  & \geq \| v \|_{H^1_{\alpha,Q}(D)}^2 + C^2 \| v \|_{L^2(D)}^2
  - 2 C \| v \|_{H^1_{\alpha,Q}(D)} \| v \|_{L^2(D)} \\
  & \geq (1-\epsilon^2) \| v \|_{H^1_{\alpha,Q}(D)}^2
  + C^2 (1-1/\epsilon^2) \| v \|_{L^2(D)}^2,
  \qquad \epsilon \in (0,1). 
\end{align*}
In consequence,
\begin{multline}
  \label{eq:almostDone}
  - \Re \left[ a_Q( v-L_k(Q\nabla v), \, v ) \right] \geq 
  \left( \frac{1-\epsilon^2}{c^2} -1 \right) \| v \|_{H^1_{\alpha,Q}(D)}^2 \\
  - \Re \int_D w\overline{v} \d{x}
  + C^2 \frac{\epsilon^2-1}{(c\epsilon)^2} \| v \|_{L^2(D)}^2
  + \Re\left[ a_Q( (L_k-L_\i)(Q\nabla v), \, v ) \right]
\end{multline}
for $\epsilon \in (0,1)$. Since $c<1$ there exists
$\epsilon \in (0,1)$ such that $1-\epsilon^2 > c^2$,
that is, $(1-\epsilon^2)/c^2 -1 > 0$.
The last three terms on the right-hand side
of~\eqref{eq:almostDone} can then be treated as compact
perturbations, in a similar way as in the proof of the first part. 
\end{proof}

\begin{remark}
  \label{th:remark}
  (a) If $\Re (Q)< -1$ in $D$, then solutions to $\div ((I_2+Q)\nabla u) + k^2 u$ 
  decay exponentially in $D$. If not only the electric permittivity but also the 
  magnetic permeability changed sign, then the corresponding solution 
  would not decay, yielding a possibly more interesting metamaterial. Volume 
  integral equations for such structures yield operator equations combining 
  $L_k$ and $V_k$, see, e.g.,~\cite{Kirsc2009}. Since $V_k$ is compact 
  on $H^1$, the above G\r{a}rding inequalities extend to this setting. For simplicity, 
  we restrict ourselves here to the non-magnetic case.
  
  (b) In the last result, we assumed that the sign of $\Re (Q)$ is constant in $D$. 
  It is possible to treat sign changes of the contrast function in $D$, but the 
  simple choice $\psi = - E(v)$ that we plugged in the second part of the proof 
  into~\eqref{eq:variationalW} has to be adapted. 
  
\end{remark}

It is a standard result that the G\r{a}rding inequalities from the last theorem imply the following  
consequences for the solvability of the integral equation and the scattering problem. 

\begin{theorem}
\label{eq:fredholm}
Suppose that the assumptions of Theorem~\ref{l-coerc_2}(a) or (b) hold, 
that the boundedness condition~\eqref{eq;boundIm} holds, 
and that the homogeneous equation $v - L_k (Q \nabla v) = 0$ in $H^1_{\alpha, Q}(D)$ 
has only the trivial solution. Then~\eqref{eq:lippmannD} has a unique solution 
for all $f \in L^2(D,\C^2)$. If $f = Q \nabla u^i$, then this solution can be extended by the 
right-hand side of~\eqref{eq:lippmannD} to a solution to the variational formulation of the 
scattering problem~\eqref{eq:variationalPeriodic}. Especially, if the integral equation is uniquely 
solvable in $H^1_{\alpha, Q}(D)$, then~\eqref{eq:variationalPeriodic} is uniquely solvable
in $H^1_\alpha(\Omega_\rho)$.
\end{theorem}

\section{G\r{a}rding Inequalities in Standard Sobolev Spaces}
\label{se:isotropic} 

The generalized G\r{a}rding inequalities from the last
section imply G\r{r}arding inequalities in the standard 
unweighted periodic Sobolev space $H^1_\alpha(D)$ 
if the material parameter $\epsr$ (or, equivalently, the 
contrast), is isotropic. Hence, in this section we assume that 
the contrast is a scalar real-valued function $q$, that is,  
\[
  Q = q I_2 \quad \text{in } \Omega.
\]  
As above, $\ol{D}$ is the support of $q$.
Under this assumption we denote the weighted 
Sobolev spaces from~\eqref{eq:weightedNorm} by $H^1_{\alpha, q}(D)$,
and their norm by 
\[
  \Vert u \Vert_{H^1_{\alpha,q}(D)}:= \left( \| \sqrt{|\Re(q)|} \nabla u \|_{L^2(D,\C^2)}^2 + \| u \|_{L^2(D)}^2 \right)^{1/2}.
\]
Since $q$ is real-valued, the form $a_q$ from~\eqref{eq:aq} 
is the inner product of $H^1_{\alpha,q}(D)$, and the 
generalized G\r{a}rding inequalities from the last section 
directly transform to standard ones.
Again, we assume that  the sign of $q$ is constant in $D$. 
Since we use regularity 
theory to prove compactness of certain commutators, we will need to  
require more smoothness of $q$ and $D$ compared to the results 
in the last section. 

\begin{lemma}
  \label{th:Compact2}
  Assume that $D$ is a domain of class $C^{2,1}$ and that
  $\mu \in C^{2,1}(\ol{D})$, $2\pi$-periodic in $x_1$. Then 
  $T: \, H^1_\alpha(D) \rightarrow H^1_\alpha(D)$ defined by 
  $Tv := \div \big[\mu V_k(q\nabla(v/\mu)) - V_k(q\nabla v)\big]$
  is a compact operator.
\end{lemma}
\begin{proof}
  We denote by $\mu^\ast \in C^{2,1}(\overline{\Omega_{\rho}})$
  a periodic extension of $\mu \in C^{2,1}(\ol{D})$ to $\Omega_{\rho}$
  (see Appendix~\ref{se:extension} on periodic extension operators).
  Then $\mu^\ast|_D = \mu$. Consider the two $\alpha$-quasi-periodic functions
  \[
    w_1 = V_k(q\nabla(v/\mu)) \quad \text{and} \quad
    w_2 = V_k(q\nabla v) \quad \text{in } \Omega_\rho.
  \]
  Both functions satisfy differential equations,
  \[
    \Delta(\mu^\ast w_1) + k^2 (\mu^\ast w_1) =
    \begin{cases}
      -q \mu \nabla(v/\mu) + 2\nabla \mu\cdot\nabla w_1 + w_1 \Delta \mu & \text{ in } D,\\
      2\nabla \mu^*\cdot\nabla w_1 + w_1\Delta\mu^* & \text{ in } \Omega_{\rho}\setminus \ol{D},
    \end{cases}
  \]
  and $\Delta w_2 + k^2w_2  = -q \nabla v$ in $D$ and 
  $\Delta w_2 + k^2w_2 = 0$ in $\Omega_{\rho}\setminus \ol{D}$.
  Hence, $w =\mu^\ast w_1 - w_2$ solves 
  \[
    \Delta w + k^2 w =
    \begin{cases}
      -q \mu \nabla(1/\mu)v + 2\nabla \mu \cdot \nabla w_1 + w_1\Delta\mu =: g_1 & \text{ in } D,\\
      w_1\Delta\mu^* + 2\nabla \mu^*\cdot\nabla w_1=: g_2 & \text{ in } \Omega_{\rho}\setminus \ol{D}.
    \end{cases}
  \]
  The functions $g_1$ and $g_2$ belong to $H^1_\alpha(D)$ and
  $H^1_\alpha(\Omega_{\rho}\setminus \ol{D})$, respectively.
  Their norms in these spaces are bounded by the norm of
  $\mu$ in $C^{2,1}(\ol{D})$ times the norm of $v$
  in $H^1_\alpha(D)$. Due to Lemma~\ref{th:H2smoothness},
  the jump of the trace and the normal trace of $w_{1,2}$
  across $\partial D$ vanishes. Hence, the Cauchy data
  of $w$ are also continuous across the boundary of $D$.

  Since the volume potential $V_k$ is bounded from $L^2(D)$
  into $H^2_\alpha(D)$, it is clear that $w$ belongs to
  $H^2_\alpha(D)$. The  smoothness assumptions on $D$ and
  $\mu$ moreover allow to apply elliptic transmission regularity
  results~\cite[Theorem 4.20]{McLea2000}
  to conclude that $w$ is even smoother than $H^2$.
  These regularity results will in turn imply the compactness
  of the operator $T: \, v \mapsto \div w$ on $H^1_\alpha(D)$.
  A straightforward adaption of the transmission regularity 
  result~\cite[Theorem 4.20]{McLea2000} 
  to the periodic setting 
  shows that  
  \[
    \|w\|_{H^3(D)}
    \leq C \left[ \|w\|_{H^1(\Omega_\rho)} + \|g_1\|_{H^1( D)}
    + \|g_2\|_{H^1(\Omega_\rho\setminus \ol{D})} \right]
    \leq C \|v\|_{H^1_\alpha(D)}.
  \]
%
%
\end{proof}

The following lemma shows that the G\r{a}rding inequalities
in the weighted spaces $H^1_{\alpha,q}(D)$ can be transformed
into estimates in $H^1_{\alpha}(D)$ if, roughly speaking,
the real-valued contrast $q$ is smooth enough and if
$(\nabla q)/q$ is bounded.

\begin{theorem}
\label{th:FreeGarding}
Assume that the scalar contrast $q$ is real-valued,
that $|q| \geq q_0 >0$ in $D$, and that 
 $\sqrt{|q|} \in C^{2,1}(\ol{D})$.
Moreover, assume that $D$ is of class $C^{2,1}$.

(a) If $q>0$ there exists a compact operator $K_+$ on $H^1_\alpha(D)$ such that
\begin{equation*}
  \Re\langle v-L_k(q\nabla v),v\rangle_{H^1_\alpha(D)} \geq \|v\|^2_{H^1_\alpha(D)}
  - \Re\langle K_+v, \, v\rangle_{H^1_\alpha(D)}, \qquad v \in H^1_\alpha(D).
\end{equation*}

(b) If $q<0$, and if
  \begin{equation}
    \label{eq:condition2}
    \| E \|_{H^1_{\alpha}(D) \to H^1_{\alpha}(\Omega_{2\rho})}  <  \inf_D | q |^{1/2}, 
  \end{equation}
then there exists a compact operator $K_-$ on $H^1_\alpha(D)$ such that
\begin{equation*}
  - \Re\langle v-L_k(q\nabla v),v\rangle_{H^1_\alpha(D)} \geq C\|v\|^2_{H^1_\alpha(D)}
  - \Re\langle K_-v, \, v\rangle_{H^1_\alpha(D)}, \qquad v \in H^1_\alpha(D),
\end{equation*}
where $C$ is the constant from~\eqref{eq:coerc_3}.
\end{theorem}
\begin{proof}
We only prove case (a) here, supposing that
$q > q_0>0$ in $D$. The proof for case
(b) is analogous, essentially one needs to
replace $\sqrt{q}$ by $\sqrt{|q|}$. For
simplicity, let us from now on abbreviate
\[
  \mu := \sqrt{q} \in C^{2,1}(\ol{D}).
\]
Choose an arbitrary $u \in H^1_\alpha(D)$ and
consider $v = u /\mu$. Our assumptions on $q$
imply that $v \in H^1_{\alpha,q}(D)$, since
$\| v \|_{H^1_{\alpha,q}(D)}^2
  \leq (2 + \| 1/\mu\|_\infty^2 + 2 \| (\nabla\mu)/\mu \|_\infty^2) \| u \|_{H^1_\alpha(D)}^2$.
In Theorem~\ref{l-coerc_2}(a) (see also Remark~\ref{remark1})
we showed that
$\Re \langle v-L_k(q\nabla v), \, v \rangle_{H^1_{\alpha,q}(D)}
  \geq \|v\|^2_{H^1_{\alpha,q}(D)} - \Re\langle  K_1 v, \, v\rangle_{H^1_{\alpha,q}(D)}$
for a compact operator $K_1$ on $H^1_{\alpha,q}(D)$.
This implies that 
\begin{multline*}
  \Re \langle u-L_k(q\nabla u), \, u \rangle_{H^1_{\alpha}(D)} 
  \geq  \| u \|^2_{H^1_\alpha(D)}
    + \Re\langle K_1(u/\mu), \, u/\mu \rangle_{H^1_{\alpha,q}(D)}\\
  + \Re\langle K_2 u, \, \nabla u \rangle_{L^2(D, \C^2)}
  + \Re\langle K_3 u, \, u \rangle_{L^2(D)},
\end{multline*}
with operators  
\begin{multline*}
 K_2 u =  \nabla\Big[\div\big[ \mu V_k(q\nabla(u/\mu)) - V_k(q\nabla u)\big]\Big] 
 - \nabla\Big[ \nabla\mu \cdot V_k(q\nabla(u/\mu))\Big] + (\nabla\mu) L_k(q\nabla(u/\mu))
\end{multline*}
and 
\begin{equation*}
  K_3 u = q\nabla(1/\mu) \cdot \big[\nabla L_k(q\nabla(u/\mu))\big]
   +  L_k(q\nabla(u/\mu))/\mu - L_k(q\nabla u).
\end{equation*}
Lemma~\ref{th:Compact2}, the smoothness of $q$, and the 
boundedness of $V_k$ and $L_k$ from $L^2(D)$ and 
$L^2(D,\C^2)$ into $H^2_\alpha(D)$ and $H^1_\alpha(D)$, respectively,
show that $K_2$ and $K_3$ are compact and bounded from
$H^1_\alpha(D)$ into $L^2(D)$, respectively. Then the 
compact embedding $H^1_\alpha(D) \subset L^2(D)$ and an application of 
Lemma~\ref{th:Compact1} imply the claim.
\end{proof}

\begin{remark}
  The regularity assumptions on $\partial D$ and $q$ can be 
  weakened using more sophisticated regularity results. It 
  is for instance possible to treat piecewise smooth $q$ by an 
  analogous technique. We do not discuss this issue to avoid 
  technicalities that would not add new ideas.
\end{remark}

\begin{theorem}
\label{eq:fredholm2}
Suppose that $q$ and $D$ satisfy the assumptions of Theorem~\ref{th:FreeGarding}(a) or (b), 
and that the homogeneous equation $v - L_k (q \nabla v) = 0$ in $H^1_{\alpha, q}(D)$ 
has only the trivial solution. Then~\eqref{eq:lippmannD} has a unique solution 
for all $f \in L^2(D,\C^2)$. If $f = q \nabla u^i$, then this solution can be extended by the 
right-hand side of~\eqref{eq:lippmannD} to a solution to the scattering 
problem~\eqref{eq:variationalPeriodic}. Especially, if the integral equation is uniquely 
solvable, then~\eqref{eq:lippmannD} is also uniquely solvable.  
\end{theorem}

\appendix

\section{Smoothness of the Difference of Periodic Green's Functions}
\label{se:appendixA}

The following lemma is a consequence of the corresponding result 
for the fundamental solution to the Helmholtz equation in free-space.

\begin{lemma}
  \label{th:analyticDifference}
  Assume that $k^2 \not= \alpha_j^2$ for all $j \in \Z$.
  Then the difference $G_{k,\alpha} - G_{\i,\alpha}$ can be written as
  \[
    G_{k,\alpha}(x) - G_{\i,\alpha}(x) = \alpha(|x|^2) + C|x|^2\ln(|x|) \beta(|x|^2)
  \]
  where $\alpha$ and $\beta$ are analytic functions and $C$ is a constant.
\end{lemma}

The smoothness of the difference $G_{k,\alpha} - G_{\i,\alpha}$ implies the 
following compactness statement for the corresponding volume potentials. 

\begin{corollary}
  \label{th:compactness}
  Assume that $k^2 \not= \alpha_j^2$ for all $j \in \Z$.
  Then $L_k - L_\i$ is compact on $H^1_{\alpha}(D)$.
\end{corollary}

\section{Periodic Extension Operators}
\label{se:extension}

In this section, we exemplary show how to 
construct a periodic extension operator 
\[
  E: \, H^1_{\alpha}(D) \to H^1_{\alpha}(\Omega),
  \quad \left. E(u) \right|_{D} = u, 
  \quad \left. E(u) \right|_{\Omega \setminus \Omega_{2\rho}} = 0, 
\]
that is used in Theorem~\ref{l-coerc_2}.
We will only construct $E$ for the case that the boundary of 
$D = \{ (x_1, x_2)^\top: \, x_1 \in (-\pi, \pi), \,  \zeta_-(x_1) < x_2 < \zeta_+(x_1) \}$
is given by two $2\pi$-periodic Lipschitz continuous functions 
$\zeta_\pm : \, \R \to (-\rho, \rho)$ such that 
$\zeta_- < - 2\rho/3$, $\zeta_+> 2\rho/3$,  and 
$| \zeta_\pm(x_1) - \zeta_\pm(x_1') | \leq M |x_1 - x_1'|$ for $x_1, x_1' \in \R$.
The general case can be tackled using local patches as in~\cite[Appendix A]{McLea2000}.

For $u \in H^1_\alpha(D)$, we define 
\[
  v(x_1, x_2) =
  \begin{cases}
    u(x_1, 2 \zeta_+(x_1) - x_2) & \text{if } \zeta_+(x_1) < x_2 < 2 \zeta_+(x_1) - \zeta_-(x_1), \\
    u(x_1, x_2) & \text{if } \zeta_-(x_1) < x_2 < \zeta_+(x_1), \\
    u(x_1,  2 \zeta_-(x_1) - x_2) & \text{if } 2 \zeta_-(x_1) - \zeta_+(x_1) < x_2 < \zeta_-(x_1).
  \end{cases}
\]
Note that $2 \zeta_+(x_1) - \zeta_-(x_1)> 2 \rho $ and that 
$2 \zeta_-(x_1) - \zeta_+(x_1) < - 2 \rho$. 
Straightforward computations show that 
$\| v \|_{H^1(\Omega_{2\rho})}\leq \max(\sqrt{3}, 2\sqrt{2}M) \| u \|_{H^1_\alpha(D)}$,
and the definition of $v$ implies that this function is $\alpha$-quasi-periodic.

To define the periodic extension operator, we use a smooth cut-off function 
$\chi: \R \to \R$, that satisfies $0\leq \chi \leq 1$, 
$\chi(x_2) = 1$ for $|x_2| \leq \rho$, and $\chi(2\rho)=0$
for $|x_2|\geq 2\rho$. Then we set 
\[
  E(u) = w, \qquad w(x) =
  \begin{cases}
    \chi(x_2) v(x) & \text{for } x \in \Omega_{2\rho},\\
    0      & \text{else}.
  \end{cases}
\]

\bibliographystyle{siam}
\bibliography{ip-biblio}

\end{document}